%% file: floerfolds.tex
   \pdfoutput=1 

\documentclass{article}
\usepackage{excludeonly}

\usepackage{savesym}
\usepackage{scrextend}  
\usepackage{amsthm,amsmath}
\usepackage{amssymb,latexsym,graphicx}
\usepackage{amscd}
 \usepackage[all,cmtip]{xy}
\usepackage{tikz-cd}
  \usetikzlibrary{arrows}
\tikzset{
  commutative diagrams/.cd,
  arrow style=tikz
  }
\usepackage{tikz}

\usepackage{accents}
\usepackage{cite}
\usepackage{mathtools}
\usepackage{stmaryrd} 

\usepackage{calligra}
\DeclareMathAlphabet{\mathcalligra}{T1}{calligra}{m}{n}
\DeclareMathAlphabet{\mathpzc}{OT1}{pzc}{m}{it}
\usepackage{hycolor}
\usepackage{xcolor}
\usepackage[
                       colorlinks=true,
                       linkcolor=black, 
                       citecolor=black, 
                       urlcolor=blue,
                     ]{hyperref}
\usepackage{soul} 
\usepackage{stackengine} 
\usepackage{booktabs} 

\input{joa-environments-english}

\input{joa-latex-shortcuts}

\begin{document}
\sloppy
\author{\quad Urs Frauenfelder \quad \qquad\qquad
             Joa Weber\footnote{
  Email: urs.frauenfelder@math.uni-augsburg.de
  \hfill
  joa@unicamp.br
  }
    \\
    Universit\"at Augsburg \qquad\qquad
    UNICAMP
}

\title{Floerfolds and Floer functions}

\date{\today}

\maketitle 

\begin{abstract}
In this article we introduce the notion of Floer function
which has the property that the Hessian
is a Fredholm operator of index zero 
in a scale of Hilbert spaces.
Since the Hessian has a complicated transformation under chart
transition, in general this is not an intrinsic condition.
Therefore we introduce the concept of Floerfolds for which we show
that the notion of Floer function is intrinsic.
\end{abstract}

\tableofcontents

\newpage
\section{Introduction}

While we nowadays have many examples of Floer homologies
the work of Floer still remains somehow mysterious.
By constructing the celebrated semi-infinite dimensional Morse
homology~\cite{floer:1988c,floer:1989c}
Floer considered a very weak metric to define the gradient.
The Hessian of such a weak metric becomes an unbounded operator.
or, if alternatively one considers a scale of Hilbert spaces,
a Fredholm operator of index zero from the smaller space to the
larger space.

In the example of loop spaces the smaller space is the space of $W^{1,2}$
loops whereas the larger space is the space of $L^2$ loops.
These two spaces together with the dense and compact inclusion
$W^{1,2}\INTO L^2$ build a scale of Hilbert spaces $(H_0,H_1)=(L^2,W^{1,2})$.
This pair can be naturally extrapolated to the triple
$(H_0,H_1,H_2)=(L^2,W^{1,2},W^{2,2})$ and
the Hessian has the regularizing property that if one restricts it
to $H_2$ it becomes as well a Fredholm operator of index zero
from $H_2$ to $H_1$.

Although the concept is already taught in basic calculus
a confusing aspect of the Hessian is its complicated transformation
under coordinate change.
In fact, it is far from obvious that the properties of the Hessian
explained above are intrinsic, i.e. independent of the choice of the chart.
The main purpose of this note is to propose
a general notion of space where the above property of the Hessian
becomes an intrinsic property.
The spaces we construct we refer to as \emph{Floerfolds}
and the functions which admit such a Hessian we refer to as
\emph{Floer functions}.

To define Floerfolds we introduce the notion of Floer map and
Floeromorphism.
Roughly speaking, a Floer map is a two times differentiable map
between level $0$ and level $2$ which as well extends to level $0$,
but also to level $-1$. The requirement that they extend to level $-1$
is probably kind of unexpected. However, since we want that our
Hessian is also a Fredholm operator from level $2$ to level $1$
and under coordinate changes the Hessian transforms with the help of
\emph{the adjoint} of the Jacobian the extension to level $-1$ seems
necessary.
In fact, in the case of the loop space level $-1$ corresponds
to $W^{-1.2}$ functions which have to be interpreted as distributions.
We show that the composition of Floer maps is again a Floer map.
This enables us to define Floerfolds via atlases whose transition
maps are Floeromorphisms.

\smallskip
The main result of this paper is Theorem~\ref{thm:Floer-function}.
In this theorem we show that pulling back a Floer function under a
Floeromorphism is again a Floer function so that the notion of
Floer function becomes an intrinsic concept on a Floerfold.
We state this main result as follows.

\begin{theoremABC}\label{thm:Floer-function}
The notion of Floer function is intrinsic.
\end{theoremABC}

\begin{proof}
Theorem~\ref{thm:pull-back}.
\end{proof}

In the last section we show how the loop space of a manifold $M$
gets endowed with the structure of Floerfold.
We show that the chart transition on the underlying manifold $M$
gives rise to a Floeromorphism between the loops in these two
different charts.

\smallskip
The motivation for having a general notion of Floer homology is the
following.
There are many properties of gradient flow lines
which should hold true in every reasonable Floer theory
like gluing or admitting the structure of a manifold with boundary and
corners under the Morse-Smale condition.
With appearance of new Floer homologies related to
Hamiltonian delay equations these general facts should be
proven in a uniform way and for that we need to figure out what
the actual structure is lying behind Floer homology.
This article makes a contribution to this endeavor.

\medskip\noindent
{\bf Acknowledgements.}
UF~acknowledges support by DFG grant
FR~2637/4-1.

\section{Floeromorphisms}

\subsection{Two- and three-level strong scale differentiability}

\begin{definition}[Two-level $\SSC^2$]\label{def:2L-ssc}
Let $(H_1,H_2)$ be a Hilbert space pair.
Let $U_1$ and $V_1$ be open subsets of $H_1$.
The part of $U_1$ in $H_2$ is open in $H_2$, 
in symbols $U_2:=U_1\cap H_2=\iota^{-1}(U_1)$ where the map
$\iota\colon H_2\to H_1$ is inclusion.
Similarly $V_2:=V_1\cap H_2$ is open in $H_2$.
We say that a map $\phi\colon U_1\to V_1$ is
\textbf{two-level strongly \boldmath$\SC^2$}, or 
\textbf{two-level $\mbf{\SSC^2}$},
if $\phi$ is $C^2$ and the restriction of $\phi$ to $U_2$
takes values in $V_2$ and as a map $\phi_2\colon U_2\to V_2$ is $C^2$.
For a two-level $\SSC^2$ map we write
$$
   \phi\colon (U_1,U_2)\to (V_1,V_2).
$$
\end{definition}

The notion of $\SSC^2$-map is due to Hofer-Wysocki-Zehnder~\cite{Hofer:2021a}.
But differently from us they consider $\SSC^2$ maps on infinitely many levels.

\begin{remark}
Let $\phi\colon U_1\to V_1$ be two-level $\SSC^2$.
Then the two maps
$$
   d^2\phi\colon U_2\times H_1\times H_1\to H_1
   ,\qquad
   U_2\to\Ll(H_1,H_1;H_1)
   ,\quad
   q\mapsto d^2\phi|_q
$$
are continuous since $\phi\in C^2(U_1,V_1)$
and inclusion $U_2\INTO U_1$ is continuous.
\end{remark}

\begin{definition}[Three-level $\SSC^2$]\label{def:3L-ssc}
Let $(H_0,H_1,H_2)$ be a Hilbert space triple.
Let $U_0$ and $V_0$ be open subsets of $H_0$.
Define open subsets $U_1:=U_0\cap H_1$ of $H_1$ and $U_2:=U_0\cap H_2$
of $H_2$; analogously define $V_1$ and $V_2$.
A \textbf{three-level \boldmath$\SSC^2$} map
is a $C^2$ map $\phi\colon U_0\to V_0$ with the property
that its restrictions to $U_1$ and $U_2$, respectively, take values in
$V_1$ and $V_2$, respectively, and as maps
$\phi_1\colon U_1\to V_1$ and $\phi_2\colon U_2\to V_2$ are $C^2$.
For a three-level $\SSC^2$ map we write
$$
   \phi\colon (U_0,U_1,U_2)\to (V_0,V_1,V_2).
$$
\end{definition}

\subsection{Floer maps}

\begin{definition}[Floer map]\label{def:Floer-map}
Let $(H_0,H_1,H_2)$ be a Hilbert space triple.
A two-level $\SSC^2$ map $\phi\colon U_1\to V_1$ between open subsets
of $H_1$
is called \textbf{\boldmath$s$-Floer map} where $s\in[0,1)$,
if it satisfies the following.
\begin{itemize}\setlength\itemsep{0ex} 
\item[(i$)_1$]
  For any $q\in U_1$ there is a continuous linear map on
  $H_0$, notation $D\phi|_q\in\Ll(H_0)$, which extends the derivative
  $d\phi_q$ from $H_1$ to $H_0$, i.e. the diagram
  \begin{equation}\label{eq:i_1}
  \begin{tikzcd} 
     H_0
     \arrow[rr, dashed, "{D\phi|_q}"]
     &&
     H_0
     \\
     H_1
     \arrow[u, hook]
     \arrow[rr, "{d\phi|_q}",  "{{\color{gray}q\in U_{1}}}"']
     &&
     H_1
     \arrow[u, hook]
  \end{tikzcd} 
  \quad,\quad
  D\phi|_q\in\Ll(H_1)\cap\Ll(H_0)
  \end{equation}
  commutes. 
  Furthermore, the map $D\phi$ defined by
  $$
     D\phi\colon U_1\to\Ll(H_0)
     ,\quad
     q\mapsto D\phi|_q
  $$
  is continuously differentiable.
\item[(i$)_2$]
  For any $q\in U_2$ the extension $D\phi|_q\in\Ll(H_0)$
  extends further to $\Ll(H_{-1})$,
  still denoted by $D\phi|_q\in\Ll(H_{-1})$.
  Furthermore, the map $D\phi$ defined by
  $$
     D\phi\colon U_2\to\Ll(H_{-1})
     ,\quad
     q\mapsto D\phi|_q
  $$
  is continuously differentiable.
\item[(ii$)_1$]
  For any $q\in U_1$ there exists a continuous bilinear map, notation
  $D^2\phi|_q\in\Ll(H_s,H_0;H_0)$, which extends 
  $d^2\phi|_q\in\Ll(H_1,H_1;H_1)$, i.e. the diagram
  \begin{equation*}
  \begin{tikzcd} 
     H_s\times H_0\arrow[rr, dashed, "D^2\phi|_q"]
     &&
       H_0
     \\
     H_1\times H_1
     \arrow[u, hook]
     \arrow[rr, "{d^2\phi|_q}",  "{{\color{gray}q\in U_{1}}}"']
     &&
     H_1
     \arrow[u, hook]
  \end{tikzcd} 
  \end{equation*}
  commutes. 
  Furthermore, the map
  $$
     D^2\phi \colon U_1\to\Ll(H_s,H_0;H_0)
     ,\quad
     q\mapsto D^2\phi|_q
  $$ %
  is continuous.
\item[(ii$)_2$]
  For any $q\in U_2$ the continuous bi-linear extension
  $D^2\phi|_q\in\Ll(H_s,H_0;H_0)$ extends,
  upon restriction of the first entry, to a continuous bi-linear map
  $D^2\phi|_q\in\Ll(H_{1+s},H_{-1};H_{-1})$.
  Furthermore, the map
  $$
     D^2\phi \colon U_2\to\Ll(H_{1+s},H_{-1};H_{-1})
     ,\quad
     q\mapsto D^2\phi|_q
  $$
  is continuous.
\end{itemize}
\end{definition}

\begin{remark}[Derivative of $D\phi$]
The derivative of $D\phi$ is related to $D^2\phi$ as follows.
If $q\in U_1$, then it is the restriction of $D^2\phi|_q\colon
H_s\times H_0\to H_0$, namely
$$
   dD\phi|_q=(D^2\phi|_q) |_{H_1\times H_0}\colon H_1\times H_0\to H_0 .
$$
To see this consider $\xi,\eta\in H_1$.
Then $dD\phi|_q(\xi,\eta)
\stackrel{\text{(i$)_1$}}{=}d^2\phi|_q(\xi,\eta)
\stackrel{\text{(ii$)_1$}}{=}D^2\phi|_q(\xi,\eta)$.
Since $H_1$ is dense in $H_0$
the identity $dD\phi|_q(\xi,\eta)=D^2\phi|_q(\xi,\eta)$
extends from $H_1\times H_1$ to $H_1\times H_0$.
\\
Given $q\in U_2$,
applying the same reasoning to $\xi,\eta\in H_2$
and using~(i$)_2$ and~(ii$)_2$ instead, the derivative
of $D\phi\colon U_2\to\Ll(H_{-1})$
is the restriction of $D^2\phi \colon U_2\to\Ll(H_{1+s},H_{-1};H_{-1})$, namely
$$
   dD\phi|_q=(D^2\phi|_q) |_{H_2\times H_{-1}}\colon H_2\times H_{-1}\to H_{-1} .
$$
\end{remark}

\begin{remark}[Definition~\ref{def:Floer-map}~(i$)_2$]
If $q\in U_2$, then we have the following commuting tower of extensions
\begin{equation}\label{eq:i_2-ext}
\begin{tikzcd} 
   H_{-1}
   \arrow[rr, dashed, "{D\phi|_q}", "{{\color{gray}q\in U_{2}}}"']
   &&
   H_{-1}
   \\
   H_0
   \arrow[u, hook]
   \arrow[rr, dashed, "{D\phi|_q}", "{{\color{gray}q\in U_{1}}}"']
   &&
   H_0
   \arrow[u, hook]
   \\
   H_1
   \arrow[u, hook]
   \arrow[rr, dashed, "{d\phi|_q}", "{{\color{gray}q\in U_{1}}}"']
   &&
   H_1
   \arrow[u, hook]
   \\
   H_2
   \arrow[u, hook]
   \arrow[rr, "{d\phi_2|_q}", "{{\color{gray}q\in U_{2}}}"']
   &&
   H_2 .
   \arrow[u, hook]
\end{tikzcd} 
\end{equation}
Since $H_2$ is dense in all three spaces $H_1$, $H_0$, and $H_{-1}$,
all three horizontal maps $D\phi|_q$ in the diagram are uniquely determined
by $d\phi_2|_q$. Furthermore, the map
$$
   D\phi \colon U_2\to
   \Ll(H_2)\cap \Ll(H_1)\cap \Ll(H_0)\cap \Ll(H_{-1})
   ,\quad
   q\mapsto D\phi|_q
$$
is continuous.
\end{remark}

\begin{remark}[Definition~\ref{def:Floer-map}~(ii$)_2$]
\mbox{}
If $q\in U_2$, then we have the commuting diagram of extensions
\begin{equation*}
\begin{tikzcd} 
   H_{1+s}\times H_{-1}
   \arrow[rrr, dashed, "{D^2\phi|_q}", "{{\color{gray}q\in U_{2}}}"']
   &&&
   H_{-1}
   \\
   &H_s\times H_0
   \arrow[rr, dashed, "{D^2\phi|_q}", "{{\color{gray}q\in U_{1}}}"']
   &&
   H_0
   \arrow[u, hook]
   \\
   &H_1\times H_1
   \arrow[u, hook]
   \arrow[rr, dashed, "{d^2\phi|_q}", "{{\color{gray}q\in U_{1}}}"']
   &&
   H_1
   \arrow[u, hook]
   \\
   &H_2\times H_2
   \arrow[uuul, hook]
   \arrow[u, hook]
   \arrow[rr, "{d^2\phi_2|_q}", "{{\color{gray}q\in U_{2}}}"']
   &&
   H_2 .
   \arrow[u, hook]
\end{tikzcd} 
\end{equation*}
Furthermore, $D^2\phi\colon q\mapsto D^2\phi|_q$, is continuous as a map
$$
   U_2
   \to
   \Ll(H_{1+s},H_{-1};H_{-1})\cap \Ll(H_s,H_0;H_0)
   \cap\Ll(H_1,H_1;H_1)\cap\Ll(H_2,H_2;H_2) .
$$
\end{remark}

\begin{remark}\label{rem:extension}
\mbox{}
\begin{itemize}\setlength\itemsep{0ex} 

\item[(a)]
  If $s_1<s_2$, then $s_1$-Floer maps are $s_2$-Floer.

\item[(b)]
Restricting a three-level $\SSC^2$-map $\phi\colon U_0\to V_0$
produces an $s$-Floer map $\phi_1\colon U_1\to V_1$ whenever $s\in[0,1)$.

\item[(c)] For $q\in U_1$ the extension $D:=D\phi|_q\in \Ll(H_0)\cap\Ll(H_1)$
in~(i$)_1$, by the Stein Weiss interpolation theorem
(see e.g.~\cite[5.4.1 p.115]{Bergh:1976a}), lies in $\Ll(H_s)$ and
$$
   \norm{D}_{\Ll(H_s)}
   \le \norm{D}_{\Ll(H_0)}^{1-s} \norm{D}_{\Ll(H_1)}^s
   ,\quad s\in[0,1].
$$
In particular, the diagram~(\ref{eq:i_1})
extends to a commutative diagram
\begin{equation}\label{eq:i_s_s}
\begin{tikzcd} 
   H_0
   \arrow[rr, dashed, "{D\phi|_q}", "{{\color{gray}q\in U_{1}}}"']
   &&
   H_0
\\
   H_s
   \arrow[rr, dashed, "{D\phi|_q|_{H_s}}", "{{\color{gray}q\in U_{1}}}"']
   \arrow[u, hook]
   &&
   H_s
   \arrow[u, hook]
\\
   H_1
   \arrow[u, hook]
   \arrow[rr, "{d\phi|_q}", "{{\color{gray}q\in U_{1}}}"']
   &&
   H_1 .
   \arrow[u, hook]
\end{tikzcd} 
\end{equation}

\item[(d)]
The extension $D:=D\phi|_q\in \Ll(H_0)$ in~(i$)_1$
is continuous as a map
$$
   D\phi\colon U_1\to \left(\Ll(H_0)\cap\Ll(H_1),\norm{\cdot}_{\max}\right),
   \quad
   q\mapsto D\phi|_q .
$$
Indeed $D\phi \colon U_1\to \Ll(H_0)$ is continuous by~(i$)_1$.
Moreover, the restriction of $D\phi|_q$ to $H_1$ equals $d\phi|_q$
which is continuous as a map $U_1\to \Ll(H_1)$ since $\phi\in C^2(U_1,V_1)$.
Furthermore, by the estimate in (c), the restriction
\begin{equation}\label{eq:(d)}
   D\phi|_{H_s}\colon U_1\to\Ll(H_s)
\end{equation}
is continuous for each $s\in[0,1]$.

\item[(e)]
By the same reasoning as in (c) and (d)
the following is true.
If $q\in U_2$, then the commuting diagram~(\ref{eq:i_2-ext})
extends to the following commuting tower of extensions
\begin{equation}\label{eq:i-ext}
\begin{tikzcd} 
   H_{-1}
   \arrow[rr, dashed, "{D\phi|_q}", "{{\color{gray}q\in U_{2}}}"']
   &&
   H_{-1}
\\
   H_0
   \arrow[u, hook]
   \arrow[rr, dashed, "{D\phi|_q}", "{{\color{gray}q\in U_{1}}}"']
   &&
   H_0
   \arrow[u, hook]
\\
   H_1
   \arrow[u, hook]
   \arrow[rr, dashed, "{d\phi|_q}", "{{\color{gray}q\in U_{1}}}"']
   &&
   H_1
   \arrow[u, hook]
\\
   H_{1+s}
   \arrow[u, hook]
   \arrow[rr, dashed, "{d\phi|_q|_{H_{1+s}}}", "{{\color{gray}q\in U_{1}}}"']
   &&
   H_{1+s}
   \arrow[u, hook]
\\
   H_2
   \arrow[u, hook]
   \arrow[rr, "{d\phi_2|_q}", "{{\color{gray}q\in U_{2}}}"']
   &&
   H_2 .
   \arrow[u, hook]
\end{tikzcd} 
\end{equation}
The restriction $d\phi|_q|_{H_{1+s}}\colon U_2\to\Ll(H_{1+s})$
is continuous for each $s\in[0,1]$.

\end{itemize}
\end{remark}

That the composition of Floer maps is again a Floer map
depends on the Stein-Weiss interpolation theorem.

\begin{proposition}[Composition]\label{prop:composition}
Let $(H_0,H_1,H_2)$ be a Hilbert space triple.
Consider $s$-Floer maps $\phi\colon U_1\to V_1$
and $\psi\colon V_1\to W_1$ between open subsets of~$H_1$. 
Then the composition $\psi\circ\phi\colon U_1\to W_1$ is an
$s$-Floer map as well.
\end{proposition}

\begin{proof}
The composition $\psi\circ\phi\colon U_1\to V_1\to W_1$
of two two-level $\SSC^2$ maps is a two-level $\SSC^2$ map,
because composing two $C^2$ maps gives a $C^2$ map,
same for the restriction $(\psi\circ\phi)_2=\psi_2\circ\phi_2\colon
U_2\to V_2\to W_2$.

\smallskip
\noindent
(i$)_1$
For $q\in U_1$ the operator defined by
$D(\psi\circ\phi)|_q:=D\psi|_{\phi(q)}\circ D\phi|_q$ extends the operator
$d(\psi\circ\phi)|_q=d\psi|_{\phi(q)}\circ d\phi|_q\in\Ll(H_1)$ to
$\Ll(H_0)$.
\\
Moreover, the map $D(\psi\circ\phi)\colon U_1\to\Ll(H_0)$,
$q\mapsto D\psi|_{\phi(q)}\circ D\phi|_q$, is continuous as both factors are.
It remains to show that the map is continuously differentiable.
For that purpose we consider the map as a composition of two maps

For the derivative the following version of the \textbf{Leibniz rule} holds
\begin{equation}\label{eq:Leibniz}
   dD(\psi\circ\phi)|_q (\xi,\eta)
   =dD\psi|_{\phi(q)} (d\phi|_q\xi, D\phi|_q\eta)
   +D\psi|_{\phi(q)}\circ dD\phi|_q(\xi,\eta)
\end{equation}
for $\xi\in H_1$ and $\eta\in H_0$.
This Leibniz rule can be deduced from the chain rule as follows.
We consider the composition
\begin{equation*}
\begin{tikzcd} [row sep=tiny] 
   U_1
   \arrow[r, "{\Ff}"]
   &
   \Ll(H_0)\times\Ll(H_0)
   \arrow[r, "{V}", "\text{bi-lin.}"']
   &
   \Ll(H_0)
   \\
   q
   \arrow[r, mapsto]
   &
   \bigl(\underbrace{D\psi|_{\phi(q)}}_{=:S},\underbrace{D\phi|_q}_{=:T}\bigr)
   \arrow[r, mapsto]
   &
   D\psi|_{\phi(q)}\circ dD\phi|_q
\end{tikzcd} 
\end{equation*}
where $V(S,T)=S\circ T$.
The derivative of $V$ is given by
$$
    dV|_{(S,T)}(\hat S,\hat T)=\hat S\circ T+S\circ\hat T
$$
and the derivative of $\Ff$ is given by
$$
   d\Ff|_q\colon H_1\to \Ll(H_0)\times\Ll(H_0)
   ,\quad
   \xi\mapsto
   \bigl(\underbrace{dD\psi|_{\phi(q)}(d\phi|_q\xi,\cdot)}_{\hat S},
   \underbrace{dD\phi|_q(\xi,\cdot)}_{\hat T}\bigr) .
$$
Thus by the chain rule the derivative exists and is of the form
\begin{equation*}
\begin{split}
   d(V\circ\Ff)|_q\xi
   &=dV|_{\Ff(q)}\circ d\Ff|_q\xi\\
   &=dD\psi|_{\phi(q)}(d\phi|_q\xi,D\phi|_q\cdot)
   +D\psi|_{\phi(q)}\circ dD\phi|_q(\xi,\cdot)
\end{split}
\end{equation*}
for any $\xi\in H_1$.
Continuity of this map in $q$ holds by axiom (i$)_1$ for $\phi$ and
for $\psi$.

\smallskip
\noindent
(i$)_2$
Same argument as in~(i$)_1$.
For $q\in U_2$ define
$D(\psi\circ\phi)|_q:=D\psi|_{\phi(q)}\circ D\phi|_q$
using the extensions to $\Ll(H_{-1})$ from (i$)_2$.

\smallskip
\noindent
(ii$)_1$
Pick $q\in U_1$, then for $\xi,\eta\in H_1$ we obtain
\begin{equation*}
\begin{split}
   &d^2(\psi\circ\phi)|_q(\xi,\eta)\\
   &=d^2\psi|_{\phi(q)}\left(d\phi|_q\xi,d\psi|_q\eta\right)
   +d\psi|_{\phi(q)} \circ d^2\phi|_q(\xi,\eta)\\
   &\stackrel{2}{=}\underbrace{D^2\psi|_{\phi(q)}}_{H_s\times H_0\to H_0}
   \Bigl(\overbrace{D\phi|_q}^{H_s\to H_s}\xi,
   \overbrace{D\psi|_q}^{H_0\to H_0}\eta\Bigr)
   +\underbrace{D\psi|_{\phi(q)}}_{H_0\to H_0}\circ
   \overbrace{D^2\phi|_q}^{H_s\times H_0\to H_0}(\xi,\eta)\\
   &=: D^2(\psi\circ\phi)|_q(\xi,\eta)
\end{split}
\end{equation*}
As indicated in the formula, equality 2 makes sense for $\xi\in H_s$
and $\eta\in H_0$. Here item~(c) of Remark~\ref{rem:extension} enters.
Inspection term by term shows that the map
$D^2(\psi\circ\phi)\colon U_1\to \Ll(H_s,H_0;H_0)$
is composed of continuous maps due to the axioms for $\phi$ and $\psi$
and, in addition, the map in~(\ref{eq:(d)}).

\smallskip
\noindent
(ii$)_2$
Pick $q\in U_2$, then for $\xi,\eta\in H_2$ we obtain
\begin{equation*}
\begin{split}
   &d^2(\psi_2\circ\phi_2)|_q(\xi,\eta)\\
   &=d^2\psi_2|_{\phi(q)}\left(d\phi_2|_q\xi,d\psi_2|_q\eta\right)
   +d\psi_2|_{\phi_2(q)} \circ d^2\phi_2|_q(\xi,\eta)\\
   &\stackrel{2}{=}\underbrace{D^2\psi|_{\phi_2(q)}}_{H_{1+s}\times H_{-1}\to H_{-1}}
   \Bigl(\overbrace{d\phi|_q}^{\Ll(H_{1+s})}\xi,
   \overbrace{D\psi|_q}^{\Ll(H_{-1})}\eta\Bigr)
   +\underbrace{D\psi|_{\phi_2(q)}}_{\Ll(H_{-1})}\circ
   \overbrace{D^2\phi|_q}^{H_{1+s}\times H_{-1}\to H_{-1}}(\xi,\eta)\\
   &=: D^2(\psi\circ\phi)|_q(\xi,\eta) .
\end{split}
\end{equation*}
As indicated in the formula, equality 2 makes sense for $\xi\in H_{1+s}$
and $\eta\in H_{-1}$.
Continuity of the map
$$
   D^2(\psi\circ\phi) \colon U_2\to\Ll(H_{1+s},H_{-1};H_{-1})
   ,\quad
   q\mapsto D^2(\psi\circ\phi)|_q
$$
follows as above by using 
Remark~\ref{rem:extension}~(e).

This proves Proposition~\ref{prop:composition}.
\end{proof}

\subsection{Floeromorphisms}

\begin{definition}[Floeromorphism]\label{def:Floeromorphism}
Let $(H_0,H_1,H_2)$ be a Hilbert space triple.
An \textbf{\boldmath$s$-Floeromorphism}
is a bijective $s$-Floer map whose inverse is an $s$-Floer map, too.
\\
By $\mathrm{Floer}_s(U_1,V_1)$ 
we denote the \textbf{set of \boldmath$s$-Floeromorphisms} from
$U_1$ to $V_1$.
\end{definition}

\begin{lemma}[Local implies global]\label{le:loc-glob}
Let $(H_0,H_1,H_2)$ be a Hilbert space triple and $s\in[0,1)$.
Let $\phi\colon U\to V$ be a homeomorphism
between open subsets of $H_1$. Assume that we have open covers
$\cup_{\beta\in A}U_\beta=U$ and $\cup_{\beta\in B}V_\beta=V$
such that for every $\beta\in B$ the map $\phi$ restricts to an
$s$-Floeromorphism $\phi|_{U_\beta}\colon U_\beta \to V_\beta$.
Then $\phi$ itself is an $s$-Floeromorphism.
\end{lemma}

\begin{proof}
This follows since derivatives are local.
\end{proof}

\section{Floerfolds}

\begin{definition}[Floer-atlas]
Let $X$ be a topological space and $s\in[0,1)$.
An \textbf{\boldmath$s$-Floer atlas} for $X$
is a collection $\Aa=\{\rho_\alpha\}_{\alpha\in A}$
of homeomorphisms $\rho_\alpha\colon X\supset V_\alpha\to
U_\alpha\subset H_1$ between open sets such that the following
conditions hold.
\begin{itemize}\setlength\itemsep{0ex} 
\item[(i)]
  $\cup_{\alpha\in A}V_\alpha= X$.
\item[(ii)]
  For any $\alpha,\beta\in A$ the map defined between open subsets
  of $H_1$ by
  $$
     \phi_{\alpha\beta}
     :=\rho_\beta\circ \rho_\alpha^{-1}|_{\rho_\alpha(V_\alpha\cap V_\beta)}
     \colon \rho_\alpha(V_\alpha\cap V_\beta)
     \to \rho_\beta(V_\alpha\cap V_\beta)
  $$
  is an $s$-Floeromorphism, called an
  \textbf{\boldmath$s$-Floer transition map}.
\end{itemize}
Two $s$-Floer atlases $\Aa=\{\rho_\alpha\}_{\alpha\in A}$ and
$\Bb=\{\rho_\beta\}_{\beta\in B}$ for $X$ are called \textbf{compatible},
notation $\Aa\sim\Bb$,
if for all $\alpha\in A$ and $\beta\in B$ the map
$\phi_{\alpha\beta}$ is an $s$-Floer transition map.
\end{definition}

\begin{theorem}\label{thm:equivalence-relation}
Compatibility is an equivalence relation for Floer atlases.
\end{theorem}

\begin{proof}
\textit{Reflexivity.} Holds by definition.
\textit{Symmetry.}
Assume that $\Aa$ is compatible with $\Bb$.
Since the inverse of a Floeromorphism is a Floeromorphism as well,
it follows that $\Bb$ is compatible with $\Aa$.
\\
\textit{Transitivity.}
Consider three $s$-Floer atlases
$\Aa=\{\rho_\alpha\}_{\alpha\in A}$, $\Bb=\{\rho_\beta\}_{\beta\in B}$,
and $\Cc=\{\rho_\gamma\}_{\gamma\in C}$ such that $\Aa$ is compatible
with $\Bb$ and $\Bb$ is compatible with $\Cc$.
We have to show that $\Aa$ is compatible with $\Cc$.
To see this let $\alpha\in A$ and $\gamma \in C$.
We need to show that
  $$
     \phi_{\alpha\gamma}
     :=\rho_\gamma\circ \rho_\alpha^{-1}|_{\rho_\alpha(V_\alpha\cap V_\gamma)}
     \colon \rho_\alpha(V_\alpha\cap V_\gamma)
     \to \rho_\gamma(V_\alpha\cap V_\gamma)
  $$
is a Floeromorphism.
For any $\beta\in B$ we have that
$
   \phi_{\alpha\gamma}|_{\rho_\alpha(V_\alpha\cap V_\beta\cap V_\gamma)}
   =\phi_{\beta\gamma}\circ\phi_{\alpha\beta}
   |_{\rho_\alpha(V_\alpha\cap V_\beta\cap V_\gamma)}
$ as a map
$
   \rho_\alpha(V_\alpha\cap V_\beta\cap V_\gamma)
   \to
   \rho_\gamma(V_\alpha\cap V_\beta\cap V_\gamma)
$
is an $s$-Floeromorphism by Proposition~\ref{prop:composition}
by compatibility $\Aa\sim \Bb$ and $\Bb\sim \Cc$.
Hence since $\cup_\beta V_\beta=X$ it follows from
Lemma~\ref{le:loc-glob} that $\phi_{\alpha\gamma}$ is an
$s$-Floeromorphism and hence $\Aa\sim \Cc$.
This proves Theorem~\ref{thm:equivalence-relation}.
\end{proof}

\begin{definition}
An \textbf{\boldmath$s$-Floerfold} is a topological space $X$
together with an equivalence class of $s$-Floer atlases.
\end{definition}

Assume that $\Aa_i$ for $i\in A$ is an arbitrary collection
of compatible $s$-Floer atlases. Then, by definition of compatibility,
the union $\cup_{i\in I}\Aa_i$ is itself an $s$-Floer atlas
which is compatible with each $\Aa_j$ for every $j\in I$.
In particular, if $\Aa$ is an $s$-Floer atlas, then the union
$\widebar\Aa:=\cup_{\Bb\sim \Aa} \Bb$ is also an $s$-Floer atlas
which is compatible with $\Aa$ and which is maximal in the sense
that, if $\Bb$ is any $s$-Floer atlas compatible with $\Aa$, then
$\Bb\subset\widebar\Aa$.
In particular, any equivalence class of $s$-Floer atlases has a
maximal representative, which by definition of maximality is unique.
Therefore, alternatively, we can define an 
\textbf{\boldmath$s$-Floerfold} as well as a
topological space endowed with a maximal $s$-Floer atlas.

\section{Floer functions}

We first define a Floer function on an open subset of $H_1$.

\begin{definition}[Floer gradient]
\label{def:Floer-gradient}
Let $H_0\supset H_1\supset H_2$ be a Hilbert space triple.
Let $f\colon H_1\supset U_1\to\R$ be a $C^2$ function
defined on an open subset $U_1$ of $H_1$.
The part of $U_1$ in $H_2$, notation 
$U_2:=U_1\cap H_2$, is an open subset of $H_2$.

\smallskip
Under these conditions a \textbf{Floer gradient} is a map
$\Nabla{} f\colon U_1\to H_0$ satisfying the following conditions.
\begin{labeling}{\texttt{(Differentiability)}}
\item[\texttt{($H_0$-gradient)}]
  If $q\in U_1$ and $\xi\in H_1$, then it holds that
  \begin{equation}\label{eq:Floer-gradient}
     df|_q\xi
     =\INNER{\Nabla{} f|_q}{\xi}_0 .
  \end{equation}

\item[\texttt{(Restriction)}]
  The restriction of $\Nabla{} f$ to $U_2$ takes values in $H_1$,
  notation $(\Nabla{}f)_2\colon U_2\to H_1$.

\item[\texttt{(Differentiability)}]
  Both maps
  \begin{equation*}
  \begin{split}
     U_1\to H_0,\quad q
     &\mapsto \Nabla{} f|_q
  \\
     U_2\to H_1,\quad q
     &\mapsto (\Nabla{} f)_2|_q
  \end{split}
  \end{equation*}
  are continuously differentiable (i.e. of class $C^1$).
\end{labeling}
\end{definition}

\begin{definition}[Floer Hessian]
\label{def:Floer-Hessian}
Let $H_0\supset H_1\supset H_2$ be a Hilbert space triple.
Let $f\colon H_1\supset U_1\to\R$ be a $C^2$ function
defined on an open subset $U_1$ of $H_1$.
The intersection $U_2:=U_1\cap H_2$ is an open subset of $H_2$.

\smallskip
Under these conditions a \textbf{Floer Hessian} is a map
$$
   A=A(f)\colon U_1\times H_1\to H_0,\quad
   (q,\xi)\mapsto A(q,\xi)=:A^q \xi
$$
such that $A^q\in\Ll(H_1,H_0)$ and which satisfies the following properties.
\begin{labeling}{\texttt{(Restriction)}}
\item[\texttt{($H_0$-Hessian)}]
  If $q\in U_1$ and $\xi,\eta\in H_1$, then it holds that
  \begin{equation}\label{eq:Floer-Hessian}
     d^2f|_q(\xi,\eta) :=d^2f(q)(\xi,\eta)=\INNER{A^q\xi}{\eta}_0 .
  \end{equation}
\item[\texttt{(Restriction)}]
  For each $q\in U_2$ the restriction of $A^q$ to $H_2$ takes values
  in $H_1$ and is bounded as a map $A^q_2\colon H_2\to H_1$.
\item[\texttt{(Continuity)}]
  Both maps
  \begin{equation*}
  \begin{split}
     U_1\to \Ll(H_1,H_0),\quad q
     &\mapsto A^q
  \\
     U_2\to \Ll(H_2,H_1),\quad q
     &\mapsto A^q_2
\end{split}
\end{equation*}
are continuous.
\item[\texttt{(Fredholm)}]
  For every $q\in U_1$ the map $A^q\colon H_1\to H_0$ is Fredholm of
  index zero. For every $q\in U_2$ the restriction $A^q_2\colon H_2\to
  H_1$ is Fredholm of index zero as well.
\end{labeling}
\end{definition}

\begin{definition}[Floer function]
\label{def:Floer-function}
We say that a function $f\colon U_1\to\R$ is \textbf{Floer} if it
admits a Floer gradient and a Floer Hessian.
\end{definition}

\begin{remark}[Symmetry]
At any $q\in U_1$
a Floer Hessian is symmetric, namely
\begin{equation*}
   \INNER{A^q \xi}{\eta}_0=\INNER{\xi}{A^q\eta}_0
\end{equation*}
for all $\xi,\eta\in H_1$.
The reason is that $d^2 f|_q (\xi,\eta)$ is symmetric in $\xi,\eta$.
\end{remark}

\begin{remark}[Floer Hessian is derivative of Floer gradient]
Let $f\colon U_1\to\R$ be a Floer function and $A^q$ its Floer
Hessian. Then there are the identities
\begin{equation}\label{eq:Floer-Hess-gradient}
   A^q=d\Nabla{}f|_q
   ,\qquad
   A^q_2=d(\Nabla{}f)_2|_q .
\end{equation}
By \texttt{(Differentiability)}
we may differentiate~(\ref{eq:Floer-gradient}),
applying~(\ref{eq:Floer-Hessian}) we get
$$
   \INNER{d\Nabla{} f|_q\eta}{\xi}_0
   =d^2f|_q(\xi,\eta)
   =\INNER{A^q\xi}{\eta}_0
$$
for all $\xi,\eta\in H_1$.
Since $H_1$ is dense in $H_0$ it follows that this equation holds true
for any $\eta\in H_1$ and $\xi\in H_0$,
from which the first identity in~(\ref{eq:Floer-Hess-gradient}) follows.
The second identity follows by the same calculation,
just start with $\xi,\eta\in H_2$ and use that $H_2\subset H_1$.
\end{remark}

\begin{theorem}[Pull-back]\label{thm:pull-back}
Let $(H_0,H_1,H_2)$ be a Hilbert space triple.
Consider a Floeromorphism $\phi\colon U_1\to V_1$ between open
subsets of $H_1$. 
Let $f\colon V_1\to\R$ be a Floer function.
Then the composition $\tilde f:=f\circ \phi\colon U_1\to \R$ is a Floer function.
\end{theorem}

\begin{proof} There are two steps.

\smallskip
\noindent
\textbf{Step~1 (Floer gradient).}
For $q\in U_1$ and $\xi\in H_1$ the chain rule yields
\begin{equation*}
\begin{split}
   d(f\circ\phi)|_q\xi
   &=df|_{\phi(q)}\circ d\phi|_q\xi\\
   &=\INNER{\Nabla{}f|_{\phi(q)}}{d\phi|_q\xi}_0\\
   &=\INNER{\Nabla{}f|_{\phi(q)}}{D\phi|_q\xi}_0\\
   &=\INNER{(D\phi|_q)^*\Nabla{}f|_{\phi(q)}}{\xi}_0 .
\end{split}
\end{equation*}
Now we define the \textbf{Floer gradient of \boldmath$f\circ\phi$} by
\begin{equation}\label{eq:Floer-grad-comp}
\boxed{
   \Nabla{} \tilde f |_q
   =\Nabla{} (f\circ\phi)|_q
   :=(D\phi|_q)^*\Nabla{}f|_{\phi(q)} \in H_0\quad \forall q\in U_1.
}
\end{equation}

\smallskip
\noindent
\texttt{($H_0$-gradient)}
This axiom holds by definition.

\smallskip
\noindent
\texttt{(Restriction)}
We need to show that the restriction of
$\Nabla{} \tilde f \colon U_1\to H_0$
to $U_2$ takes values in $H_1$, in symbols
$(\Nabla{}\tilde f )_2 \colon U_2\to H_1$.
For this purpose take $q\in U_2$.
Then $\Nabla{}f|_{\phi(q)} \in H_1$
since $\phi(q)\in V_2$ and $f$ is a Floer function.
By (i$)_2$ in Definition~\ref{def:Floer-map}
the hypothesis of Corollary~\ref{cor:adjoints}
is satisfied and the conclusion is
\begin{equation}\label{eq:gdhj39}
   (D\phi|_q)^*\in\Ll(H_1) .
\end{equation}
Hence $\Nabla{} \tilde f |_q
=(D\phi|_q)^*\Nabla{}f|_{\phi(q)} \in H_1$
whenever $q\in U_2$.

\smallskip
\noindent
\texttt{(Differentiability)}
Level\,\,1:
We need to show that the map $U_1\to H_0$, $q\mapsto \Nabla{} \tilde f |_q
=(D\phi|_q)^*\Nabla{}f|_{\phi(q)}$, is $C^1$ (continuously differentiable).
Differentiating~(\ref{eq:Floer-grad-comp}) at a point $q\in U_1$
with the help of the Leibniz rule, which follows as in~(\ref{eq:Leibniz}),
we obtain the formula
$$
   d\Nabla{}\tilde f|_q\xi
   =d(D\phi)^*|_q\left(\xi,\Nabla{} f|_{\phi(q)}\right)
   +(D\phi|_q)^*\circ d\Nabla{} f|_{\phi(q)}\circ d\phi|_q\xi
$$
for every $\xi\in H_1$.
Because by assumption $D\phi\colon U_1\to\Ll(H_0)$ is $C^1$
and $*\colon \Ll(H_0)\to\Ll(H_0)$ is linear,
the map $(D\phi)^*\colon U_1\to\Ll(H_0)$ is $C^1$, too.
Since both maps $q\mapsto \phi(q)\mapsto \Nabla{} f|_{\phi(q)}$
are continuous, the first summand in the displayed formula
is continuous. The second summand is a composition of
a $C^1$ map, a $C^0$ map, and a $C^1$ map, thus $C^0$ itself.

\smallskip
\noindent
Level\,\,2:
We need to show that the map
$$
   U_2\to H_1,\quad
   q\mapsto(\Nabla{}\tilde f)_2 |_q
   =(D\phi|_q)^*(\Nabla{}f)_2|_{\phi_2(q)}
$$
is $C^1$. Differentiating this map at $q\in U_2$ we obtain the formula
$$
   d(\Nabla{}\tilde f)_2|_q\xi
   =d(D\phi)^*|_q\left(\xi,(\Nabla{} f)_2|_{\phi_2(q)}\right)
   +(D\phi|_q)^*\circ d(\Nabla{} f)_2|_{\phi_2(q)}\circ d\phi_2|_q\xi
$$
for every $\xi\in H_2$.
Because by assumption (i$)_2$ the map
$D\phi\colon U_2\to\Ll(H_{-1})$ is $C^1$
and $*\colon \Ll(H_{-1})\to\Ll(H_{-1}^*)=\Ll(H_1)$ is linear,
it follows that the map $(D\phi)^*\colon U_1\to\Ll(H_1)$ is $C^1$.
Since both maps $q\mapsto \phi_2(q)\mapsto (\Nabla{} f)_2|_{\phi_2(q)}$
are continuous, the first summand in the displayed formula
is continuous. 
The second summand is a composition of
a $C^1$ map, a $C^0$ map, and a $C^0$ map, thus $C^0$ itself.

\medskip
\noindent
\textbf{Step~2 (Floer Hessian).}
We need to define a map
$$
   \tilde A=A(f\circ \phi)\colon U_1\times H_1\to H_0,\quad
   (q,\xi)\mapsto \tilde A(q,\xi)=:\tilde A^q \xi
$$
such that $A^q\in\Ll(H_1,H_0)$ and verify the four axioms in
Definition~\ref{def:Floer-Hessian}.

\smallskip
\noindent
\texttt{($H_0$-Hessian)}
  For $q\in U_1$ the chain rule yields
  $
     d(f\circ\phi)|_q
     =df|_{\phi(q)}\circ d\phi|_q
  $,
  so
  \begin{equation}\label{eq:tilde-A-pre}
  \begin{split}
     &d^2(f\circ\phi)|_q(\xi,\eta)\\
     &=d^2f|_{\phi(q)} \left( d\phi|_q \xi, d\phi|_q \eta\right)
     +df|_{\phi(q)}\circ d^2\phi|_q(\xi,\eta)\\
     &=\INNER{A^{\phi(q)} d\phi|_q \xi}{d\phi|_q \eta}_0
     +\INNER{\Nabla{} f|_{\phi(q)}}{d^2\phi|_q(\xi,\eta)}_0
  \end{split}
  \end{equation}
  for all $\xi,\eta\in H_1$ where $A=A(f)$ and $\Nabla{} f$ is
  the $H_0$-gradient. Let $\iota_s\colon H_1\to H_s$ be inclusion.
  We define the \textbf{Floer Hessian of \boldmath$f\circ\phi$} by
    \begin{equation}\label{def:Floer-Hessian-pullback}
\boxed{
     \tilde A^q=A(f\circ \phi)^q
     :={D\phi|_q}^*\circ A^{\phi(q)}\circ d\phi|_q +K^q \circ \iota_s
     \colon H_1\to H_0
}
  \end{equation}
  where by the theorem of Riesz there exists a unique operator $K^q$ such that
    \begin{equation}\label{eq:K^q}
     K^q\in\Ll(H_s,H_0)
     ,\quad
     \INNER{K^q\xi}{\eta}_0
     =\INNER{\Nabla{} f|_{\phi(q)}}{D^2\phi|_q(\xi,\eta)}_0 
{\color{gray}\,
     =:B^q(\xi,\eta)
     ,
}
  \end{equation}
  for all $\xi\in H_s$ and $\eta\in H_0$.
  If $\xi,\eta\in H_1$, then by~(\ref{eq:tilde-A-pre})
  the identity~(\ref{eq:Floer-Hessian}) for
  $\tilde A$ and $\tilde f$ holds true and this proves
  \texttt{($H_0$-Hessian)} for $\tilde A$.

\smallskip
\noindent
\texttt{(Restriction)}
  For each $q\in U_2$ the restriction of $\tilde A^q$ to $H_2$ takes values
  in $H_1$ and is bounded as a map $\tilde A^q_2\colon H_2\to H_1$.
  To see this pick $q\in U_2$. Then $\phi(q)\in V_2$,
  hence $A^{\phi(q)}=A^{\phi(q)}(f)$ maps $H_2$ to $H_1$ as $f$ is
  a Floer function, so
  $$
     \underbrace{{D\phi|_q}^*}_{H_1\stackrel{\text{(\ref{eq:gdhj39})}}{\to} H_1}
     \circ \underbrace{A^{\phi(q)}}_{H_2\to H_1}
     \circ \underbrace{d\phi|_q}_{H_2\to H_2}
     \in \Ll(H_2,H_1) .
  $$
  This proves that summand one in~(\ref{def:Floer-Hessian-pullback})
  lies in $\Ll(H_2,H_1)$. Concerning summand two
  we next show that
  \begin{equation}\label{eq:K^q-U_2}
     K^q\in \Ll(H_{1+s},H_1)\cap \Ll(H_s,H_0)
     ,\quad
     \forall q\in U_2 .
  \end{equation}
  To see this pick $\xi\in H_{1+s}$ and $\eta\in H_{-1}$.
  By density of $H_0$ in $H_{-1}$ pick a sequence
  $(\eta_\nu)\subset H_0$ converging in $H_{-1}$ to $\eta$.
  By~\cite[App.\,A.3]{Frauenfelder:2024d} we have 
  the insertion isometry $\flat\colon H_1\to H_{-1}^*$,
  $\Nabla{} f|_{\phi(q)}\mapsto\INNER{\Nabla{} f|_{\phi(q)}}{\cdot}_0$.
  Then
  \begin{equation*}
  \begin{split}
     \abs{\INNER{K^q\xi}{\eta_\nu }_0}
     &\le\norm{\flat \Nabla{} f|_{\phi(q)}}_{H_{-1}^*}
     \norm{D^2\phi|_q(\xi,\eta_\nu )}_{-1}
     \\
     &\le \norm{\Nabla{} f|_{\phi(q)}}_1
     \norm{D^2\phi|_q}_{\Ll(H_{1+s},H_{-1};H_{-1})}
     \norm{\xi}_{1+s}\norm{\eta_\nu }_{-1}
  \end{split}
  \end{equation*}
  for every $\nu$.
  Take the limit $\nu\to\infty$ to obtain the estimate
  $$
   \abs{\INNER{K^q\xi}{\eta}_0}
   \le \underbrace{\norm{\Nabla{} f|_{\phi(q)}}_1
   \norm{D^2\phi|_q}_{\Ll(H_{1+s},H_{-1};H_{-1})}}_{=:\kappa_{1+s}}
   \norm{\xi}_{1+s}\norm{\eta }_{-1}
  $$
for any $\xi\in H_{1+s}$ and $\eta\in H_{-1}$.
Hence we see that the element $\INNER{K^q\xi}{\cdot}_0$ of $H_0^*$ is
even an element of $H_{-1}^*$ whose norm is bounded by
$\norm{\INNER{K^q\xi}{\cdot}_0}_{H_{-1}^*}\le \kappa_{1+s} \norm{\xi}_{1+s}$.
Using the isometric identification of $H_{-1}^*$ with $H_1$,
see~\cite[App.\,A.3]{Frauenfelder:2024d},
we see that $K^q\xi$ is an element of $H_1$ of norm
$\norm{K^q\xi}_1\le \kappa_{1+s} \norm{\xi}_{1+s}$.
Therefore $K^q$ is a bounded linear operator $H_{1+s}\to H_1$
whose operator norm is bounded by $\kappa_{1+s}$.
Such argument will reappear in
the proof of Lemma~\ref{le:bi-lin-K}.
\\
Abbreviating by $\iota_{1+s}\colon H_2\to H_{1+s}$ the inclusion
we conclude from~(\ref{def:Floer-Hessian-pullback}) that
$\tilde A^q$ restricts to an operator
\begin{equation}\label{def:Floer-Hessian-pullback-2}
\boxed{
   \tilde A^q_2
   ={D\phi|_q}^*\circ A^{\phi(q)}_2\circ d\phi_2|_q +K^q \circ \iota_{1+s}
   \in\Ll(H_2,H_1) .
}
\end{equation}

\smallskip
\noindent
\texttt{(Continuity)}
  We need to show that both maps
  \begin{equation*}
  \begin{split}
     U_1\to \Ll(H_1,H_0),\quad q
     &\mapsto \tilde A^q
  \\
     U_2\to \Ll(H_2,H_1),\quad q
     &\mapsto \tilde A^q_2
\end{split}
\end{equation*}
are continuous.
By Definition~\ref{def:Floer-map} (i$)_1$
the map $D\phi\colon U_1\to \Ll(H_0)$ is continuous.
Since taking adjoints is continuous as a map $\Ll(H_0)\to \Ll(H_0)$
we conclude that in~(\ref{def:Floer-Hessian-pullback})
the first term is continuous as a map $(D\phi)^*\colon U_1\to\Ll(H_0)$.
The map $A^{\phi(q)}\colon U_1\to\Ll(H_1,H_0)$ is continuous by
\texttt{(Continuity)} of Floer Hessians. 
The map $d\phi\colon U_1\to\Ll(H_1)$ is continuous since $\phi$ is
$\SSC^2$, thus $C^2$.
Hence the composition ${D\phi|_\cdot}^*\circ A^{\phi(\cdot)}\circ
d\phi|_\cdot\colon U_1\to\Ll(H_1,H_0)$ is continuous.

It remains the second summand $K^q\circ \iota_s$
in~(\ref{def:Floer-Hessian-pullback}). It suffices to show
that the map $K\colon U_1\to\Ll(H_s,H_0)$, $q\mapsto K^q$, is continuous.
To see this we first show that the bi-linear form
$B\colon U_1\to \Ll(H_s,H_0;\R)$, $q\mapsto B^q$, is continuous.
This follows from the fact that, by Definition~\ref{def:Floer-map}
(ii$)_1$, the map 
$$
   D^2\phi \colon U_1\to\Ll(H_s,H_0;H_0)
   ,\quad
   q\mapsto D^2\phi|_q
$$
is continuous. Moreover, by continuity of $\phi$ and
\texttt{(Differentiability)} of the Floer gradient the map
$U_1\to V_1\to H_0$, $q\mapsto\phi(q)\mapsto \Nabla{} f|_{\phi(q)}$,
is continuous. Therefore $q\mapsto B^q$ is continuous.
Since the Riesz map which associates to a bi-linear form a linear map
is itself continuous in the bi-linear form,
we conclude that the map $K\colon U_1\to \Ll(H_s,H_0)$,
$q\mapsto K^q$, is continuous as well.
This finishes the proof that $q\mapsto \tilde A^q$ is continuopus as a
map $U_1\to\Ll(H_1,H_0)$.

\smallskip
Continuity of the second map $q\mapsto \tilde A^q_2$:
Consider the first summand in~(\ref {def:Floer-Hessian-pullback-2}).
The map $q\mapsto d\phi_2\in\Ll(H_2)$ is continuous since $\phi_2\in C^2$ and
$U_2\ni q\mapsto\phi(q)\mapsto A^{\phi(q)}_2\in \Ll(H_2,H_1)$ is
continuous since $\phi_2\in C^2$ and by \texttt{(Continuity)} of Floer Hessians. 
The map $U_2\to\Ll(H_{-1})\to\Ll(H_1)$, $q\mapsto T:=D\phi|_q\mapsto T^*$,
is continuous by Definition~\ref{def:Floer-map} (i$)_2$
and by continuity of taking the adjoint.
\\
Consider the second summand $K^q\circ\iota_{1+s}$
in~(\ref {def:Floer-Hessian-pullback-2}).
Here $\iota_{1+s}\in\Ll(H_2,H_{1+s})$ is inclusion.
It remains to show that the map $U_2\mapsto \Ll(H_{1+s},H_1)$,
$q\mapsto K^q$, is continuous.
To see this we show that the bi-linear map $U_2\mapsto
\Ll(H_{1+s},H_{-1};\R)$, $q\mapsto B^q$, see~(\ref{eq:K^q}), is continuous.
By definition of the bi-linear form $B^q$ this follows from
continuity of the map $U_2\to \Ll(H_{1+s},H_{-1};H_{-1})$,
$q\mapsto D^2\phi|_q$, according to Definition~\ref{def:Floer-map} (ii$)_2$
and continuity of the map $U_2\to V_2\to H_1=H_{-1}^*$,
$q\mapsto \phi(q)\mapsto \Nabla{}f|_{\phi(q)}$, 
by continuity of $\phi_2$ and \texttt{(Differentiability)} of $\Nabla{}f$.
This proves continuity of $U_2\ni q\mapsto \tilde A^q_2\in\Ll(H_2,H_1)$.

\smallskip
\noindent
\texttt{(Fredholm)}
  Let $q\in U_1$.
  Then the first summand in~(\ref{def:Floer-Hessian-pullback}) is a
  Fredholm operator of index zero, since this is true for
  $A^{\phi(q)}\colon H_1\to H_0$ and both operators $d\phi|_q\in\Ll(H_1)$
  and ${D\phi|_q}^*\in\Ll(H_0)$ are isomorphisms.
  Concerning the second summand note that inclusion $\iota_s\colon
  H_1\to H_s$ is compact 
  due to the assumption $s<1$
  and, since $K$ is bounded, the composition $K\iota_s$ is compact as well.
  Since the Fredholm property as well as the index are stable under
  compact perturbation we conclude that the sum $\tilde A^q$ is a
  Fredholm operator of index zero as well.

  It remains to show that $\tilde A^q_2$ is also a Fredholm operator
  of index zero whenever $q\in U_2$.
  In view of formula~(\ref{def:Floer-Hessian-pullback-2})
  this follows by the same reasoning.
\end{proof}

\subsubsection*{Adjoints and bi-linear maps used in the proof}

\begin{lemma}\label{le:adjoints}
Let $(H_0,H_1)$ be a Hilbert space pair.
For $T\in \Ll(H_1)$ we denote by $T^*\in \Ll(H_1^*)$ the $H_1$-adjoint
of $T$. Then the following is true
$$
   T\in\Ll(H_0)\cap \Ll(H_1)
   \qquad\Rightarrow\qquad
   T^*\in\Ll(H_0^*)\cap \Ll(H_1^*) .
$$
\end{lemma}

\begin{proof}
To see that $T^*\in\Ll(H_0^*)$ pick $v_0^*\in H_0^*$.
Since $T\in\Ll(H_0)$ it also has an $H_0$-adjoint $T^{*_0}\in \Ll(H_0^*)$.
We claim that
\begin{equation}\label{eq:jhghgh38}
 T^* v_0^*=T^{*_0} v_0^*|_{H_1} .
\end{equation}
To see this pick $v_1\in H_1$.
Using the definition of $H_1$- and then $H_0$-adjoint we obtain
$(T^*v_0^*) v_1=v_0^*(T v_1)=(T^{*_0} v_0^*) v_1$.
This proves~(\ref{eq:jhghgh38}).

Since $H_1$ is dense in $H_0$ 
it follows from~(\ref{eq:jhghgh38}) that
$T^*v_0^*$ uniquely extends to a bounded linear map $H_0\to\R$
which coincides with $T^{*_0}v_0^*$.
In particular, $T^*v_0^*$ lies in $H_0^*$ and we have the identity
$T^* v_0^*=T^{*_0}v_0^*$ in $H_0^*$.
Since $v_0^*$ was an arbitrary element of $H_0^*$
we obtain that
$$
   T^*|_{H_0^*} = T^{*_0} \in \Ll(H_0^*).
$$
This proves that the $H_1$-adjoint
$T^*$ is an element of $\Ll(H_0^*)\cap \Ll(H_1^*)$.
\end{proof}

\begin{corollary}\label{cor:adjoints}
Under the hypotheses of Lemma~\ref{le:adjoints} it holds
$$
   T\in\Ll(H_0)\cap \Ll(H_{-1})
   \qquad\Rightarrow\qquad
   T^*\in\Ll(H_0)\cap \Ll(H_1) .
$$
\end{corollary}

\begin{proof}
Lemma~\ref{le:adjoints} together with the isometries
$H_0\simeq H_0^*$ and $H_1\simeq H_{-1}^*$
where the latter isometry stems from~\cite[App.\,A.3]{Frauenfelder:2024d}.
\end{proof}

\begin{lemma}\label{le:bi-lin-K}
Let $(H_0,H_1)$ be a Hilbert space pair and
$B\colon H_0\times H_0\to\R$ a continuous bi-linear map. By the
theorem of Riesz there is a well defined operator $K\in\Ll(H_0)$ such that
$$
   B(\xi,\eta)=\INNER{K\xi}{\eta}_0 .
$$
Suppose that there is a constant $\kappa >0$ such that
$$
   \abs{B(\xi,\eta)}\le \kappa \norm{\xi}_1\cdot\norm{\eta}_{-1}
$$
for all $\xi\in H_1$ and $\eta\in H_0$.
In this case $K$ restricts to a bounded linear operator on $H_1$, in
symbols $K\in\Ll(H_1)$.
\end{lemma}

\begin{proof}
Let $\xi\in H_1$ and $\eta\in H_0$.
By hypothesis
$$
   \abs{\INNER{K\xi}{\eta}_0}
   \le\abs{B(\xi,\eta)}
   \le \kappa \norm{\xi}_1\cdot\norm{\eta}_{-1} .
$$
We define a continuous bi-linear map as follows
$$
   \INNER{K\xi}{\cdot}_0\colon H_{-1}\to\R
   ,\quad
   \eta\mapsto \lim_{\nu\to\infty} \INNER{K\xi}{\eta_\nu}_0
$$
where $(\eta_\nu)\subset H_0$ is a sequence converging in $H_{-1}$ to
$\eta\in H_{-1}$.
Hence we see that the element $\INNER{K\xi}{\cdot}_0$ of $H_0^*$ is
even an element of $H_{-1}^*$ whose norm is bounded by
$\norm{\INNER{K\xi}{\cdot}_0}_{H_{-1}^*}\le \kappa \norm{\xi}_1$.
Using the isometric identification of $H_{-1}^*$ with $H_1$,
see~\cite[App.\,A.3]{Frauenfelder:2024d},
we see that $K\xi$ is an element of $H_1$ of norm
$\norm{K\xi}_1\le \kappa \norm{\xi}_1$.
Therefore $K$ is a bounded linear operator $H_1\to H_1$
whose operator norm is bounded by $\kappa $.
\end{proof}

\section{The loop space as a Floerfold}

In this section we show that Floerfolds naturally arise
in the object of main interest in Floer theory,
namely the free loop space.

For any manifold we show that the space of small loops
has the structure of a Floerfold.
By a small loop we mean a loop whose image fits in a single chart.
It should be possible to give similarly the full loop space the structure
of a Floerfold by decomposing the loop into several pieces each of
which fits into a single chart.
To avoid technicalities we concentrate here on small loops.

\medskip
Consider open subsets $\Uu,\Vv\subset \R^n$ and the Hilbert space triple
$$
   H_0:=L^2(\SS^1,\R^n),\qquad
   H_1:=W^{1,2}(\SS^1,\R^n),\qquad
   H_2:=W^{2,2}(\SS^1,\R^n).
$$
Define open subsets
\begin{equation*}
\begin{split}
   U_\ell:=\{u\in H_\ell\mid u(t)\in\Uu\; \forall t\in\SS^1\}
   &\subset C^0(\SS^1,\Uu),\quad \ell=1,2,
\\
   V_\ell:=\{v\in H_\ell\mid v(t)\in\Vv\; \forall t\in\SS^1\}
   &\subset C^0(\SS^1,\Vv),\quad \ell=1,2 .
\end{split}
\end{equation*}
Given a diffeomorphism $\Phi=(\Phi_1,\dots,\Phi_n)\colon\Uu\to \Vv$,
we define an $\SSC^2$-diffeomorphism
$$
   \phi
   \colon H_1\supset U_1\to V_1\subset H_1,\quad
   u\mapsto \Phi\circ u
   =(\Phi_1(u(\cdot)),\dots, \Phi_n(u(\cdot)))
$$
whose components are maps
$\phi_i=\Phi_i(u(\cdot))\colon U_1\to W^{1,2}(\SS^1,\R)$.

\begin{theorem}\label{thm:loop-Floerfold}
$\phi\colon U_1\to V_1$ is an $s$-Floeromorphism
whenever $s\in(\frac12,1)$.
\end{theorem}

\begin{proof}
Fix $s\in(\frac12,1)$.
It suffices to show that $\phi$ is an $s$-Floer map:
Interchanging the roles of $\Uu$ and $\Vv$ and applying the result to
$\Phi^{-1}$ then shows that $\phi^{-1}$ is also an $s$-Floer map, so
that $\phi$ is an $s$-Floeromorphism.
The proof that $\phi$ is an $s$-Floer map takes 4 steps.

\smallskip\noindent
\textbf{Step 1.}
We show (i$)_1$.

\begin{proof}
The first derivative of the diffeomorphism $\phi$ at $u\in U_1$ in direction
$\xi=(\xi_1,\dots,\xi_n)\in H_1$ is given by the formula
\begin{equation}\label{eq:gjhbghj77}
   d\phi|_{u}\xi
   =d\Phi|_{u(\cdot)}\xi(\cdot)
   =\biggl(\sum_{j=1}^n \p_j\Phi_1|_{u(\cdot)} \xi_j(\cdot),\dots,
      \sum_{j=1}^n \p_j\Phi_n|_{u(\cdot)} \xi_j(\cdot)\biggr)
\end{equation}
at any time $t\in\SS^1$ and where $\p_j:=\frac{\p}{\p x_j}$.
The facts that $d\phi|_{u}\xi$ lies in $H_1$ and $\xi\mapsto
d\phi|_{u}\xi$ is linear and bounded follow since, firstly,
pre-composition of $W^{1,2}$-maps with smooth maps
takes values in $W^{1,2}$, more precisely
$$
   C^\infty(\R^n,\R)\times W^{1,2}(\SS^1,\R^n)\to
   W^{1,2}(\SS^1,\R),\quad (\Psi,u)\mapsto \Psi\circ u .
$$
and, secondly, multiplication is well-defined and continuous 
as a map
\begin{equation}\label{eq:12-12-12}
   W^{1,2}(\SS^1,\R)\times W^{1,2}(\SS^1,\R)\to W^{1,2}(\SS^1,\R)
   ,\quad (g,h)\mapsto gh .
\end{equation}
The latter relies on continuity of inclusion
$W^{1,2}(\SS^1,\R)\INTO C^0 (\SS^1,\R)$.
This shows that
$
   d\phi|_{u}\in\Ll(H_1)
$.
Moreover, since multiplication
\begin{equation}\label{eq:0-2-2}
   C^0 (\SS^1,\R)\times L^2 (\SS^1,\R)\to L^2 (\SS^1,\R),\quad (g,h)\mapsto gh
\end{equation}
thus, due to $W^{1,2}\INTO C^0$, multiplication
\begin{equation}\label{eq:12-2-2}
   W^{1,2}(\SS^1,\R)\times L^2 (\SS^1,\R)\to L^2 (\SS^1,\R),\quad (g,h)\mapsto gh
\end{equation}
is continuous,
the map $d\phi|_{u}$ has a unique extension to $\Ll(H_0)$, notation
$$
   D\phi|_u\in\Ll(H_0) .
$$
Here $D\phi|_u\xi$ is defined again by the right hand side of
equation~(\ref{eq:gjhbghj77}).
In particular for any $\xi\in H_1$ both maps coincide
$D\phi|_u\xi=d\phi|_u\xi$.
Similarly, the map
$$
   D\phi\colon U_1\to \Ll(H_0),\quad
   u\mapsto D\phi|_{u}
$$
is continuous.
Summarizing we have the picture
$$
   \underbrace{\p_j\Phi_i|_{u(t)}}_{W^{1,2}}
   \underbrace{\xi_j(t)}_{L^2}\in L^2.
$$

It remains to show that $D\phi\colon U_1\to \Ll(H_0)$ is continuously
differentiable. Given $u\in U_1$ and $\xi\in H_1$, using~(\ref{eq:gjhbghj77}) we
compute the derivative $dD\phi|_u\colon H_1\to\Ll(H_0)$ as follows
\begin{equation}\label{eq:2nd-deriv}
\begin{split}
   &\bigl(\left(dD\phi|_{u}\xi\right)\eta\bigr)(t)\\
   &=\biggl(\sum_{j,k=1}^n \underbrace{\overbrace{\p_k\p_j\Phi_1|_{u(t)}}^{W^{1,2}}
      \overbrace{\xi_j(t)}^{W^{1,2}}}_{W^{1,2}} \underbrace{\eta_k(t)}_{L^2},\dots,
      \sum_{j,k=1}^n \p_k\p_j\Phi_n|_{u(t)} \xi_j(t)\eta_k(t)\biggr) .
\end{split}
\end{equation}
Since multiplication of functions~(\ref{eq:12-12-12})
and~(\ref{eq:12-2-2}) are continuous maps
the derivative is a well defined map $H_1\to\Ll(H_0)$ and depends
continuously on $u\in U_1$.
This shows Step~1.
\end{proof}
\noindent
\textbf{Step 2.}
We show (i$)_2$.

\begin{proof}
By the same arguments as in Step~1,
but using the multiplication Lemma~\ref{le:-12_12_-12} instead,
we see that $D\phi$ already on $U_1$ extends to $\Ll(H_{-1})$, namely
\begin{equation}\label{eq:2nd-derivi}
   \underbrace{\p_j\Phi_i|_{u(t)}}_{W^{1,2}}
   \underbrace{\xi_j(t)}_{W^{-1,2}}\in W^{-1,2}
   ,\qquad
   \underbrace{\overbrace{\p_k\p_j\Phi_1|_{u(t)}}^{W^{1,2}}
   \overbrace{\xi_j(t)}^{W^{1,2}}}_{W^{1,2}}
   \underbrace{\eta_k(t)}_{W^{-1,2}} \in W^{-1,2}.
\end{equation}
Then a-fortiori $D\phi$
gives rise to a $C^1$ map
$D\phi\colon U_2\INTO U_1\to\Ll(H_{-1})$.
\end{proof}
\noindent
\textbf{Step 3.}
We show (ii$)_1$.

\begin{proof}
As we already computed in~(\ref{eq:2nd-deriv}) we have for $u\in U_1$
the formula
\begin{equation*}
\begin{split}
   &\left(d^2\phi|_{u}(\xi,\eta)\right)(t)\\
   &=\biggl(\sum_{j,k=1}^n \p_k\p_j\Phi_1|_{u(t)} \xi_j(t) \eta_k(t),\dots,
      \sum_{j,k=1}^n \p_k\p_j\Phi_n|_{u(t)} \xi_j(t)\eta_k(t)\biggr)
\end{split}
\end{equation*}
for all $\xi,\eta\in H_1$ and times $t\in\SS^1$.
It is a side remark that, differently from the first derivative,
while multiplication of two $W^{1,2}$ functions is still in $W^{1,2}$,
multiplication of two $L^2$ functions is only in $L^1$.

As a consequence of Proposition~\ref{prop:Taylor} the
bi-linear map $d^2\phi|_{u}\in \Ll(H_1,H_1;H_1)$
extends uniquely to a bi-linear map $D^2\phi|_{u}\in \Ll(H_s,H_0;H_0)$
whenever $s\in (\frac12, 1]$. To see this observe the inclusions
$$
   \underbrace{\underbrace{\p_k\p_j\Phi_1|_{u(\cdot)}}_{\in W^{1,2}\subset C^0}
   \underbrace{\xi_j(t)}_{\in H_s\subset C^0}}_{\in C^0}
   \underbrace{\eta_k(t)}_{\in H_0}
   \in H_0=L^2\;\;\text{by~(\ref{eq:0-2-2})}.
$$
Moreover, the map
$$
   D^2\phi\colon U_1\to \Ll(H_s,H_0;H_0),\quad
   u\mapsto D^2\phi|_{u}
$$
is continuous.
\end{proof}
\noindent
\textbf{Step 4.}
We show (ii$)_2$.

\begin{proof}
By the second equation in~(\ref{eq:2nd-derivi})
we see that $D^2\phi \colon U_1\to\Ll(H_s,H_0;H_0)$
after restriction extends to a continuous map
$D^2\phi \colon U_1\to\Ll(H_1,H_{-1};H_{-1})$.

A-fortiori, by continuous inclusions $U_2\INTO U_1$ and
$H_{1+s}\INTO H_1$
the derivative $D^2\phi$ becomes a continuous map
$U_2\to\Ll(H_{1+s},H_{-1};H_{-1})$.
This proves Step~4 and Theorem~\ref{thm:loop-Floerfold} follows.
\end{proof}
\end{proof}

\subsubsection*{Sobolev theory used in the proof}\label{sec:app-loop}

\begin{lemma}\label{le:-12_12_-12}
Let $W^{-1,2}(\SS^1)$ be the dual space of $W^{1,2}(\SS^1)$.
Then $W^{-1,2}(\SS^1)$ is preserved by $W^{1,2}$ multiplication,
more precisely, multiplication gives a map
\begin{equation}\label{eq:mult}
   \cdot\;\colon W^{-1,2}(\SS^1)\times W^{1,2}(\SS^1)\to W^{-1,2}(\SS^1)
   ,\quad
   (f^*,g)\mapsto f^*\cdot g
\end{equation}
and this map is continuous.
\end{lemma}

Observe that $f^*\cdot g\colon  W^{1,2}(\SS^1)\to\R$ is a linear
functional and evaluation is given by $(f^*\cdot g)h=f^*(gh)$
for every $h\in W^{1,2}(\SS^1)$.

\begin{proof}
Pick $h\in W^{1,2}(\SS^1)$, then $(f^*\cdot g)(h)=f^*(gh)$.
Since $W^{1,2}$ is closed under multiplication,
the product $gh$ lies in $W^{1,2}$ and therefore $f^*\in (W^{1,2})^*$
and $f^*(gh)\in\R$. Therefore $f^*\cdot g$ is a linear map $W^{1,2}\to\R$.
By continuity of the multiplication map~(\ref{eq:12-12-12})
there is a constant $c$ such that the next estimate holds
$$
   \abs{f^*(gh)}\le \norm{f^*}_{-1,2} \norm{gh}_{1,2}
   \le c \norm{f^*}_{-1,2} \norm{g}_{1,2}\norm{h}_{1,2} .
$$
This shows that $f^*g$ is continuous as a map $W^{1,2}\to\R$.
In particular $f^*g \in W^{-1,2}=(W^{1,2})^*$.
Moreover, we have the estimate
$$
   \norm{f^*g}_{-1,2}
   \le c \norm{f^*}_{-1,2}\norm{g}_{1,2} ,
$$
Hence the map~(\ref{eq:mult}) is continuous.
\end{proof}

\begin{proposition}\label{prop:Taylor}
The inclusion map
$H_s(\R^m)\subset C^\alpha(\R^m)$ is continuous
whenever $s=\frac{m}{2}+\alpha$ and $\alpha\in(0,1)$.
Here
$$
   u\in C^\alpha(\R^m)\quad\Leftrightarrow\quad
   \text{$u$ bounded and $\exists C\colon
   \abs{u(x+y)-u(x)}\le C\abs{y}^\alpha$ $\forall x,y$}.
$$
\end{proposition}

\begin{proof}
See e.g.~\cite[Ch.\,4 Prop.\,1.5]{Taylor:1996a}.
\end{proof}

\bibliographystyle{alpha}
\addcontentsline{toc}{section}{References}
\bibliography{$HOME/Dropbox/0-Libraries+app-data/Bibdesk-BibFiles/library_math,$HOME/Dropbox/0-Libraries+app-data/Bibdesk-BibFiles/library_math_2020,$HOME/Dropbox/0-Libraries+app-data/Bibdesk-BibFiles/library_physics}{}

%


\end{document}

%% file: joa-environments-english.tex
\newtheorem{theoremABC}{Theorem}

\newtheorem{theorem}{Theorem}[section]

\newtheorem{corollary}[theorem]{Corollary}

\newtheorem{lemma}[theorem]{Lemma}
\newtheorem{proposition}[theorem]{Proposition}

\theoremstyle{definition}
\newtheorem{definition}[theorem]{Definition}

\newtheorem{remark}[theorem]{Remark}

\theoremstyle{remark}



%% file: joa-latex-shortcuts.tex
\newcommand{\R}{{\mathbb{R}}}
\renewcommand{\SS}{{\mathbb{S}}}

\newcommand{\Aa}{{\mathcal{A}}}   
\newcommand{\Bb}{{\mathcal{B}}}
\newcommand{\Cc}{{\mathcal{C}}}   

\newcommand{\Ff}{{\mathcal{F}}}

\newcommand{\Ll}{{\mathcal{L}}}   

\newcommand{\Uu}{{\mathcal{U}}}
\newcommand{\Vv}{{\mathcal{V}}}

\newcommand{\cgraph}[1]{\Gamma_{\kern-.5ex{}#1}}     

\newcommand{\norm}{{\rm norm}}

\newcommand{\INNER}[2]{\left\langle #1, #2\right\rangle}

\newcommand{\SC}{{\mathrm{sc}}}                  
\newcommand{\SSC}{{\mathrm{ssc}}}               
\newcommand{\mbf}[1]{\text{\boldmath $#1$}}  

\def\NABLA#1{{\mathop{\nabla\kern-.5ex\lower1ex\hbox{$#1$}}}}
\def\Nabla#1{\nabla\kern-.5ex{}_{#1}}
\def\Tabla#1{\Tilde\nabla\kern-.5ex{}_{#1}}
\def\Babla#1{\widebar\nabla\kern-.5ex{}_{#1}}
\def\abs#1{\mathopen|#1\mathclose|}   
            
\def\norm#1{\mathopen\|#1\mathclose\|}

\renewcommand{\Tilde}{\widetilde}

\newcommand{\p}{{\partial}}

\newcommand{\INTO}{\hookrightarrow}              

\newlength\eqshift
\renewcommand\theequation{\thesection.\arabic{equation}}
\let\savetheequation\theequation

\makeatletter
\renewcommand*\env@matrix[1][\arraystretch]{%
  \edef\arraystretch{#1}%
  \hskip -\arraycolsep
  \let\@ifnextchar\new@ifnextchar
  \array{*\c@MaxMatrixCols c}}
\makeatother

\makeatletter
\let\save@mathaccent\mathaccent
\newcommand*\if@single[3]{%
  \setbox0\hbox{${\mathaccent"0362{#1}}^H$}%
  \setbox2\hbox{${\mathaccent"0362{\kern0pt#1}}^H$}%
  \ifdim\ht0=\ht2 #3\else #2\fi
  }
\newcommand*\rel@kern[1]{\kern#1\dimexpr\macc@kerna}
\newcommand*\widebar[1]{\@ifnextchar^{{\wide@bar{#1}{0}}}{\wide@bar{#1}{1}}}
\newcommand*\wide@bar[2]{\if@single{#1}{\wide@bar@{#1}{#2}{1}}{\wide@bar@{#1}{#2}{2}}}
\newcommand*\wide@bar@[3]{%
  \begingroup
  \def\mathaccent##1##2{%
    \let\mathaccent\save@mathaccent
    \if#32 \let\macc@nucleus\first@char \fi
    \setbox\z@\hbox{$\macc@style{\macc@nucleus}_{}$}%
    \setbox\tw@\hbox{$\macc@style{\macc@nucleus}{}_{}$}%
    \dimen@\wd\tw@
    \advance\dimen@-\wd\z@
    \divide\dimen@ 3
    \@tempdima\wd\tw@
    \advance\@tempdima-\scriptspace
    \divide\@tempdima 10
    \advance\dimen@-\@tempdima
    \ifdim\dimen@>\z@ \dimen@0pt\fi
    \rel@kern{0.6}\kern-\dimen@
    \if#31
      \overline{\rel@kern{-0.6}\kern\dimen@\macc@nucleus\rel@kern{0.4}\kern\dimen@}%
      \advance\dimen@0.4\dimexpr\macc@kerna
      \let\final@kern#2%
      \ifdim\dimen@<\z@ \let\final@kern1\fi
      \if\final@kern1 \kern-\dimen@\fi
    \else
      \overline{\rel@kern{-0.6}\kern\dimen@#1}%
    \fi
  }%
  \macc@depth\@ne
  \let\math@bgroup\@empty \let\math@egroup\macc@set@skewchar
  \mathsurround\z@ \frozen@everymath{\mathgroup\macc@group\relax}%
  \macc@set@skewchar\relax
  \let\mathaccentV\macc@nested@a
  \if#31
    \macc@nested@a\relax111{#1}%
  \else
    \def\gobble@till@marker##1\endmarker{}%
    \futurelet\first@char\gobble@till@marker#1\endmarker
    \ifcat\noexpand\first@char A\else
      \def\first@char{}%
    \fi
    \macc@nested@a\relax111{\first@char}%
  \fi
  \endgroup
}
\makeatother

\long\def\symbolfootnote[#1]#2{\begingroup%
\def\thefootnote{\fnsymbol{footnote}}\footnote[#1]{#2}\endgroup}

\tikzset{
  symbol/.style={
    draw=none,
    every to/.append style={
      edge node={node [sloped, allow upside down, auto=false]{$#1$}}}
  }
}

